\documentclass[11pt]{article}

\usepackage[top=50mm, bottom=50mm, left=50mm, right=50mm]{geometry}

\usepackage{amssymb}
\usepackage{amsmath}
\usepackage{amsthm}

\usepackage[displaymath,mathlines]{lineno}

\usepackage{epsfig}
\usepackage{graphicx}
\usepackage{graphics}
\usepackage{float}
\usepackage{subfigure}
\usepackage{multirow}
\usepackage{hyperref}
\usepackage{xcolor}\hypersetup{colorlinks,
	linkcolor={red!50!black},
	citecolor={blue!50!black},
	urlcolor={blue!80!black}
}

\usepackage{fullpage}
\usepackage[normalem]{ulem} 
\usepackage{makeidx}
\usepackage{xspace}
\usepackage{wrapfig}
\usepackage{dsfont}
 
\usepackage[square,numbers]{natbib}
 
 \usepackage[title]{appendix}
 
\DeclareMathOperator{\diver}{div}
\DeclareMathOperator{\sign}{sign}
\newcommand \commentout[1] {}
\newcommand\R{\mathbb{R}}
\newcommand{\A}{\mathcal{A}}
\newcommand{\B}{\mathcal{B}}
\newcommand{\C}{\mathcal{C}}
\newcommand{\D}{\mathcal{D}}

\newcommand{\p}{p_\gamma}
\newcommand{\n}{n_\gamma}
\newcommand{\cgam}{c_\gamma}
\newcommand{\dt}{dt}
\theoremstyle{plain} 
\newtheorem{thm}{Theorem}[section] 
\newtheorem{prop}[thm]{Proposition}

\newtheorem{definition}[thm]{Definition}

\title{Free boundary limit of tumor growth model with nutrient}

\author{Noemi David\thanks{Sorbonne Universit{\'e}, Inria, CNRS, Universit\'{e} de Paris, Laboratoire Jacques-Louis Lions UMR7598, F-75005 Paris. 
		Email : noemi.david@ljll.math.upmc.fr}
	\and
	Beno\^\i t Perthame\thanks{Sorbonne Universit{\'e}, CNRS, Universit\'{e} de Paris, Inria, Laboratoire Jacques-Louis Lions UMR7598, F-75005 Paris. 
		Email : Benoit.Perthame@sorbonne-universite.fr}
}
\date{\today}
\begin{document}
\maketitle
	\begin{abstract} 
		Both compressible and incompressible porous medium models are used in the literature to describe the mechanical properties of living tissues. These two classes of models can be related using a  stiff pressure law. In the incompressible limit, the compressible model generates a free boundary problem of Hele-Shaw type where incompressibility holds in the saturated phase. 
		
		Here we consider the case with a nutrient. Then, a badly coupled system of equations describes the cell density number and the nutrient concentration. For that reason, the derivation of the free boundary (incompressible) limit was an open problem, in particular a difficulty is to establish the so-called \textit{complementarity relation} which allows to recover the pressure using an elliptic equation.  To establish the limit, we use two new  ideas. The first idea, also used recently for related problems, is to extend the usual  Aronson-B\'enilan estimates in $L^\infty$ to an $L^2$ setting. The second idea is to derive a sharp uniform $L^4$ estimate on the pressure gradient, independently of space dimension.
		
	\end{abstract} 
	\vskip .7cm
	
	\noindent{\makebox[1in]\hrulefill}\newline
	2010 \textit{Mathematics Subject Classification.} 35B45; 35K57;   35K65; 35Q92; 76N10;  76T99; 
	\newline\textit{Keywords and phrases.} Porous medium equation; Tumor growth; Aronson-B\'enilan estimate; Free boundary; Hele-Shaw problem; 
	
\section*{Introduction}
	We consider a compressible mechanical model of tumor growth, where the cell motion is driven by the pressure gradient according to Darcy's law. The cell proliferation is governed by a bio-mechanical form of contact inhibition, that prevents cell division when the total cell density exceeds a critical threshold. 
	The evolution of the cell population density $n\geq 0$ and the concentration of nutrients $c\geq 0$ are described by the following type of system 
	\begin{equation}\label{sys}
	\begin{cases}
	\partial_t n - \diver(n \nabla p) = n G(p, c),\qquad \qquad x \in \R^d, \;  t\geq 0,\\[5pt]
	\partial_t c - \Delta c + n H(c) = (c_B-c)K(p),\\[5pt]
	c(x,t) \rightarrow c_B \quad \text{for }  x \rightarrow \infty.
	\end{cases}
	\end{equation}
	The pressure within the tissue is denoted by $p$, and in the compressible setting, we use for simplicity the following law of state
	\begin{equation}\label{pressure law}
	p= n^\gamma, \qquad  \gamma > 1.
	\end{equation}
	The reaction term $G(p,c)$ is the cell division rate and the lowest value of pressure that prevents cell division is called \textit{homeostatic pressure}, and we denote it by $p_H$. 
	The concentration $c_B>0$ is the level of nutrients at the source, namely the network of blood vessels. Here, we consider the vascular phase of tumor growth, after \textit{angiogenesis} has occurred, therefore the vasculature is present both outside and inside the tumor.
	The term $K\geq 0$ is the rate of nutrient release, which decreases with respect to the pressure. As clinical observation have shown, the mechanical stress generated by the cells shrinks the vessels inside the tumor and effects the blood flow and, by consequence, the nutrients  delivery, see \cite{MMACCL} for further details. Finally, the term $H\geq 0$ is an increasing function of $c$ and represents the consumption rate of the nutrient by the tumor cells.
	
	The specific form of the reaction term in the equation on $c$ is not fully relevant for our analysis, and we only need the possibility to derive some generic a priori estimates, mostly in $L^2$.  Our study covers, for example, the terms  in \cite{PQV} where the authors take $H=H(p,c)$, $K=0$ and those in \cite{PTV} where $K=\mathds{1}_{\{n=0\}}$, since the authors are considering the avascular phase of tumor growth. For our study, only some general  conditions are needed, which are detailed in the next sections.
	
\paragraph*{Motivation and previous works.}
	Models of tumor growth, including \eqref{sys}, possibly with more biological relevance, have been widely used recently. Several surveys are available, as \cite{RCM}. Numerical schemes for the model at hand, with AP property (asymptotic preserving), have been proposed in \cite{LTWZ}. 

	Mechanical models of tumor growth are focused on the effect of the internal pressure which governs the dynamics of the cell population density. This kind of description was initiated in \cite{Green} by Greenspan and further developed by Byrne and Chaplain, \cite{BC96}, Friedman, \cite{Fri}, and Lowengrub et al., \cite{LFJCMWC}, among the others.
	The leading assumption is that the birth of a cell generates a mechanical stress on the surrounding cells which start to move under a gradient of pressure. By consequence, the motion of the cells is usually described by Darcy's law
	\begin{equation}
	\label{Darcy}
	\vec{v}=-\nabla p,
	\end{equation}
	which relates the velocity to the pressure gradient. This type of models have been extensively used to describe the early stage of tumor growth, the so-called \textit{avascular phase}, see for example \cite{BCGR,byrneChaplain96,SCh}. 
	Models of tumor growth that include the effect of viscosity, \cite{PV, DeSc, RJPJ}, or more than one species of tissue cells, \cite{CFSS,LLP}, are also well-developed. 
	For a comprehensive review on this topic we refer the reader to \cite{Fri,LFJCMWC,PT,RSCB}.
	
	The equation for the density in the system \eqref{sys} is based on the continuous mechanical model presented in \cite{BD}, in which the dynamics of tumor growth are governed by competition for space and contact inhibition. The \textit{homeostatic pressure} is determined by the maximum level of stress that the cells can tolerate, see \cite{BD} for further details on the individual-based model that leads to the continuous one. 
	
	As explained above, this type of models are usually referred to as \textit{compressible}, since they relate the density and the pressure through a compressible constitutive law, in a fluid mechanical view. 
	A second class of models commonly used to describe tumor growth are free boundary problems, \cite{frihie}.
	They are also called geometric or \textit{incompressible} models and describe the tumor as a moving domain where the density is constant. Free boundary problems arise also from the theory of mixture applied to tumor growth, \cite{BKMP2,BP}.
	
	Building a link between these two classes of models has attracted the attention of many researchers in recent years. This result has first been achieved in \cite{PQV} for a purely mechanical model, passing to the so-called \textit{incompressible limit}, as the pressure becomes stiff. Later, it has been studied for a lot of models, which included 
viscosity \cite{PV,DeSc}, different laws of state \cite{DHV} and more than one species of cells \cite{BPPS}.
	In each case the limit model turns out to be a free boundary model of Hele-Shaw type. 
	
	Our goal is to study the limit $\gamma \to \infty$ in the law of state \eqref{pressure law}, and prove that the limit solution satisfies a free boundary problem. It has been proved in~\cite{PQV} that (the norms are specified in the next section and we now use the notation $n_\gamma, \,p_\gamma, \, \cgam$ in place of $n, \, p, \, c$ to indicate the dependency upon~$\gamma$)
	\[
	n_\gamma \to n_\infty, \qquad p_\gamma \to p_\infty, \qquad c_\gamma \to c_\infty,
	\]
	and the limits satisfy the system
	\begin{equation}\label{syslim}
	\begin{cases}
	\partial_t n_\infty - \diver(n_\infty \nabla p_\infty) = n_\infty G(p_\infty, c_\infty),\qquad \qquad x \in \R^d, \;  t\geq 0,\\[5pt]
	\partial_t c_\infty - \Delta c_\infty + n_\infty H(c_\infty) = (c_B-c_\infty)K(p_\infty),\\[5pt]
	c_\infty(x,t) \rightarrow c_B \quad\text{ for } x \rightarrow \infty,
	\end{cases}
	\end{equation}
	with a \textit{graph} relation between $p_\infty$ and $n_\infty$ given by  
	\begin{equation}\label{pressurelilmit}
	0\leq n_\infty \leq 1, \qquad  p_\infty (n_\infty-1)=0  .
	\end{equation}
	A remarkable result is the uniqueness of the weak solutions of this system.
	
	However, it was left open in~\cite{PQV} to establish the so-called \textit{complementarity condition}, which reads (in the sense of distributions)
	\begin{equation}\label{crd}
	p_\infty \big(\Delta p_\infty + G(p_\infty, c_\infty) \big)=0 \quad\text{ in }\quad \D'(\R^d\times (0,\infty)),
	\end{equation}
	which follows formally from the equation on $n$ written for the pressure, namely
	\begin{equation}\label{pressure}
	\partial_t \p = \gamma \p \big(\Delta \p + G(\p,\cgam) \big) +|\nabla \p|^2.
	\end{equation}
	The complementarity condition is fundamental because it relates the weak solutions defined by the equations \eqref{syslim} and \eqref{pressurelilmit} to the geometric form of the Hele-Shaw problem, where the set ${\cal O}(t):=\{x;\, p(x,t) >0\}$ evolves with the speed determined by the normal component of $\nabla p_\infty$. The limit pressure is a solution to the elliptic equation with Dirichlet boundary conditions
	\[
	-\Delta p_\infty = G(p_\infty , c_\infty ) \qquad \text{in }\; {\cal O}(t)=\{x;\, p_\infty(x,t) >0\}.
	\]
	The Hele-Shaw problem is a widely studied free boundary model. Although we are only interested in the weak formulation, the regularity of the boundary is also a challenging issue, see \cite{SalsaCaffa,HS,MePeQu}.
	\paragraph*{Difficulties and strategies.}
	To handle this problem, we make use of two new estimates  which hold because the cell population density satisfies a Porous Medium Equation, which reads
	\begin{equation}
	\label{pme}
	\partial_t n_\gamma -\frac{\gamma}{\gamma +1}\Delta n_\gamma^{\gamma+1} = n_\gamma G(p_\gamma,c_\gamma).
	\end{equation}
	$\bullet$ The first estimate results from the famous Aronson-B\'enilan (AB in short)  inequalities for the porous media, \cite{AB, CrPi1982}, which have been extended in various contexts (see \cite{VaVi} for another example). It was used in the purely mechanical case, \cite{PQV}, and it gives the lower bound $\Delta p_\gamma(t) +G( p_\gamma(t)) \geq - C/\gamma t$, with $C$ positive constant. Here, unlike in the case without nutrients, it cannot hold. In fact, as shown in \cite{PTV}, where a semi-explicit travelling wave solution was found, there exists a region where $p$ is constantly equal to zero and $G$ is negative.
	
	Therefore, we show a weaker, but still sufficient, condition
	\[
	\int_0^T \int_{\R^d} |\min (0, \Delta p_\gamma )|^3 \leq C(T).
	\]
	This is proved by working in $L^2$, rather than with a sub-solution, as it has been recently initiated in \cite{GPS, BPPS}. This method has the advantage to be compatible with the $L^2$ estimates on $c_\gamma$ and its derivatives. We recall here that $\Delta p_\infty$ is a bounded measure due to the free boundary of the set ${\cal O}(t)$ where the pressure is positive.
	\\
	$\bullet$ The second new estimate is an $L^4$ bound on $\nabla p_\gamma$, independent of the dimension $d$. In the simple case, where $G$ depends only on $p$, it results from the kinetic energy relation combined to the AB inequality in $L^\infty$, which is wrong here. We have a new and more general way to derive it, independently of the AB inequality.
	\paragraph*{Plan of the paper.}
	The paper is organized as follows. The next section is devoted to explain the notations and assumptions and to state the main result of the paper, namely that the complementarity condition holds. The rest of the paper is dedicated to prove this result. We begin in Section~\ref{Sec:est1} presenting standard bounds which are useful for deriving the main new estimates that are stated and proved in Section~\ref{Sec:est2}. Finally, in Section~\ref{Sec:4} we give the proof of the complementarity relation. 
	
\section{Notations, assumptions and main result}\label{sec:NAMR}
	%
	%
	\paragraph*{Notations.}
	We denote $Q=\R^d \times (0,\infty)$, and for $T>0$ we set $Q_T=\R^d\times (0,T)$. We frequently use the abbreviation form $n(t):=n(x,t),\, p(t):=p(x,t),\, c(t):=c(x,t)$.  We denote 
	\begin{equation*}
	\sign_+{\{w\}}=\mathds{1}_{\{w>0\}}\quad \text{ and } \quad \sign_-{\{w\}}=-\mathds{1}_{\{w<0\}}.
	\end{equation*}
	We also define the positive and negative part of $w$ as follows
	\begin{equation*}
	|w|_+ :=
	\begin{cases}
	w, &\text{ for } w>0,\\
	0, &\text{ for } w\leq 0,
	\end{cases}
	\quad  \text{ and }\quad
	|w|_- :=
	\begin{cases}
	-w, &\text{ for } w<0,\\
	0, &\text{ for } w\geq 0.
	\end{cases}
	\end{equation*}
	%
	%
\paragraph*{Assumptions.}
	Considering the growth/reaction terms,  the functions $G$, $H$ and $K$ are assumed to be smooth and
	we make the following assumption. There exist positive constants $\beta$, $p_H$, $p_B$ (reference pressure of blood vessels) such that
	\begin{align}\label{ReTe}
	\partial_p G &<-\beta, & \partial_c G&\geq 0, & G(p,c_B)\leq 0,& \text{ for } p\geq p_H,\\
	K'(p) &\leq 0,    & 0 \leq K(p) &\leq 1,  & \quad K(p)=0,& \text{ for } p\geq p_B,\\
	H'(c) &\geq 0,   & 0 \leq H(c) &,  & H(0)=0 ,& 
	\end{align}
	Furthermore, for a given pressure $p$, $G(p,c)<0$ for $c$ small enough. This assumption indicates that the tumor cells die by \textit{necrosis} when the concentration of nutrients is below a survival threshold. 
	
	Some standard choices for the reaction terms are
	\begin{align*}
	G(p,c)&=g(p)(c+c_1)-c_2, & H(c)&=c, & K(p)&=\left|1-\frac{p}{p_B}\right|_+,
	\end{align*}
	where $c_{1,2}$ are positive constants and $g$ is a decreasing function of $p$, see \cite{CBCB, MMACCL, PQV}. 
	%
	%
	\paragraph*{Initial data.} The system \eqref{sys} is completed with initial data. We assume that for some $n^0, c^0$, the initial data $n^0_\gamma, \cgam^0$ satisfy
	\begin{align}
	\label{inf n}  &0\leq n^0_\gamma \leq n_H:=p_H^{1/\gamma}, &\|n^0_\gamma-n^0\|_{L^1(\R^d)}\xrightarrow[\gamma    \rightarrow \infty]{}0, \qquad \qquad &n^0 \in L_+^1(\R^d), \\
	\label{inf c}  &0\leq \cgam^0\leq c_B,  &\|c^0_\gamma-c^0\|_{L^1(\R^d)}\xrightarrow[\gamma  \rightarrow \infty]{}0,\qquad \qquad & c^0-c_B \in L_+^1(\R^d).
	\end{align}
	We also assume that there is a positive constant $C$ such that
	\begin{align}
	\label{2 deltap}         \|\nabla p^0_\gamma\|_{L^2(\R^d)} + \|\Delta p^0_\gamma\|_{L^2(\R^d)} & \leq C,\\
	\label{timezero} \|(\partial_t \n)^0\|_{L^1(\R^d)}+\|(\partial_t \cgam)^0\|_{L^1(\R^d)}&\leq C,\\
	\label{2 gradc}        \|\nabla \cgam^0\|_{L^2(\R^d)} & \leq C.
	\end{align}
	Set these conditions on the initial data, we give the definition of weak solution of the system \eqref{sys} as follows.
	\begin{definition}
		Given $T>0$, a weak solution of the system \eqref{sys} is a triple $(\n,\p,\cgam)$ such that, 
		\[
		\n, \p, \cgam \in L^{\infty}((0,T),L^p(\R^d)) \quad \forall p\geq 1, \qquad  \nabla \cgam , \; \nabla \p \in  L^{2}(\R^d \times (0,T)), 
		\]
		and for all $\varphi\in C^1_{\text{comp}}(\R^d \times [0,T))$,
		\begin{align*}
		&\int_0^T \int_{\R^d}\left(-\n \partial_t \varphi + \n \nabla\p\nabla\varphi-\n G(\p,\cgam)\varphi\right)=\int_{\R^d} \n^0 \varphi(0),\\
		&\int_0^T \int_{\R^d}\left(-c_\gamma \partial_t \varphi + \nabla c_\gamma \nabla\varphi +\n H(\cgam)\varphi - (c_B-c) K(p)\varphi\right)=\int_{\R^d} c_\gamma^0 \varphi(0).
		\end{align*}
	\end{definition}  
	From \cite{V} we know that a weak solution exists for all $T>0$.
	%
	%
\paragraph*{Compact support.}
	Because our arguments rely on technical calculations, we first simplify the setting assuming that there exists a smooth bounded  open domain $\Omega_0\subset \R^d$, independent of~$\gamma$, such that for all $\gamma > 1$
	$$
	\text{supp}(\n^0)\subset \Omega_0.
	$$
	Unlike the solutions of the heat equation, the PME's solutions have a finite speed of propagation, see \cite{V}. This means that, for all $T>0$, there exists a smooth bounded open domain $\Omega_T$ independent of $\gamma$ such that
	$$\
	\text{supp}(\n(t)) \subset \Omega_T, \qquad \forall t  \in [0,T],
	$$
	see Appendix~\ref{appA} for the proof.
	From now on, we  consider a solution $(\n,\p)$ with compact support for all $\gamma > 1$. In the Appendix~\ref{app}, we show how to extend the result  to more general solutions.
	\paragraph*{Main result.} We now state the main result of the paper, namely the weak formulation of the complementarity relation.
	\begin{thm}[Estimates and complementarity relation]\label{CR1}
		With all the previous assumptions, the limit pressure $p_\infty$ satisfies the relation~\eqref{crd}, that means, 
for all test functions $\zeta \in \D(Q)$, we have
\begin{equation*}
		\iint_{Q}  \left( -|\nabla p_\infty|^2\zeta - p_\infty\nabla p_\infty\nabla\zeta + p_\infty G(p_\infty,c_\infty)\zeta\right)=0.
		\end{equation*}
Furthermore the following estimates hold uniformly in $\gamma$	
\[
\int_0^T \int_{\Omega_T} | \Delta \p + G(\p,\cgam)|_-^3 \leq C(T), \qquad  \qquad  \int_0^T\int_{\Omega_T} |\nabla \p|^4 \leq C(T).
\]
	\end{thm}
	%
	%
\section{Preliminary Estimates}
	\label{Sec:est1}
	
	Let $(\n, \p, \cgam)$ be a weak solution to the system \eqref{sys}.
	We recall some standard preliminary bounds on $\n, \p, \cgam$ and their derivatives, gathered in the following Proposition.
	%
	%
	\begin{prop}[Direct estimates]
		\label{prop1}
		Given $(\n, \p, \cgam)$ a weak solution of the system \eqref{sys} for $\gamma > 1$, and $T>0$, there exists a constant $C(T)$, independent of $\gamma$, such that for all $0\leq t\leq T$
		\begin{align}
		&0\leq \n \leq n_H, & &0\leq \p \leq p_H, & &0\leq \cgam \leq c_B,\\
		\label{L1bds}   
		&\|\n(t)\|_{L^1(\R^d)}\leq C(T), & &\|\p(t) \|_{L^1(\R^d)}\leq C(T), & &\|\cgam(t)-c_B\|_{L^1(\R^d)} \leq C(T), \\
		\label{2 c}  &\|\nabla\cgam(t)\|_{L^2(\R^d)} \leq C(T), &  &\|\Delta\cgam\|_{L^2(Q_T)} \leq C(T), &  &\|\partial_t \cgam \|_{L^2(Q_T)} \leq C(T),\\
		\label{timeder}
		&\|\partial_t \n \|_{L^1(Q_T)} \leq C(T),&  &\|\partial_t \p \|_{L^1(Q_T)} \leq C(T),&  &\|\partial_t \cgam \|_{L^1(Q_T)} \leq C(T), \\
		\label{c4p2}   &\|\nabla\cgam\|_{L^4(Q_T)} \leq C(T), & &\|\nabla\p\|_{L^2(Q_T)} \leq C(T).
		\end{align}
	\end{prop}
	For the sake of completeness, we now recall the derivation of these bounds.
	
	%
	%
\paragraph*{$L^\infty$ bounds for $\n, \p, \cgam$.} 
	The $L^\infty$ bounds are just consequences of our assumptions on $G$ using the comparison principle. 
	%
	%
\paragraph*{$L^1$ bounds on $\n, \p, \cgam$.}
	These are also standard estimates, noting that 
	\[
	\| p(t) \|_{L^1(\R^d)} = \| n(t) p(t)^\frac{\gamma- 1}{\gamma}  \|_{L^1(\R^d)} \leq p_H^\frac{\gamma- 1}{\gamma} \| n(t) \|_{L^1(\R^d)}. 
	\]
	\commentout{Arguing as before, thanks to the assumptions \eqref{inf n} and \eqref{inf c}, we can control the $L^1$ norms of both the negative and positive part of $\n,\p$ and $\cgam-c_B$. Hence
		\begin{equation*}
		\| \n(t)\|_{L^1(\R^d)}, \| \p(t) \|_{L^1(\R^d)}, \|\cgam(t)-c_B\|_{L^1(\R^d)}\leq C(T), 
		\end{equation*}
		for every $0\leq t\leq T.$
	}
	%
	%
\paragraph*{$L^2$ bounds for the derivatives of $\cgam$.}
	We now prove the $L^2$ bounds for $\nabla \cgam, \Delta \cgam$ and $\partial_t \cgam$. We multiply the equation for $\cgam$ by $-\Delta \cgam$ and we integrate in space and time
	\begin{equation*}
	- \int_0^t\int_{\R^d} \partial_t \cgam \Delta \cgam +\int_0^t\int_{\R^d} |\Delta \cgam|^2 = \int_0^t\int_{\R^d}(\n H(\cgam)-(c_B-c)K(\p))\Delta \cgam.
	\end{equation*}
	Integrating by parts and using Young's inequality we obtain
	\begin{equation*}
	\int_0^t \int_{\R^d} \partial_t(\nabla \cgam) \nabla \cgam +\int_0^t\int_{\R^d} |\Delta \cgam|^2 \leq \int_0^t\int_{\R^d}\frac{|\n H(\cgam)-(c_B-c)K(\p)|^2}{2}+ \int_0^t\int_{\R^d}\frac{|\Delta \cgam|^2}{2}.
	\end{equation*} 
	Hence, we have
	\begin{equation*}
	\frac{1}{2}\int_{\R^d} |\nabla \cgam(t)|^2  +\frac{1}{2}\int_0^t\int_{\R^d} |\Delta \cgam|^2 \leq C \int_0^t \left(\|\n(s)\|^2_{L^1(\R^d)}+\|\cgam(s)-c_B\|^2_{L^1(\R^d)}\right)ds+ \frac{1}{2}\|\nabla \cgam^0\|^{2}_{L^2({\R^d})},
	\end{equation*} 
	where $C$ is a positive constant depending on $n_H, c_B$ and the $L^\infty$ norms of $H$ and $K$.
	
	Finally, using the $L^1$ bounds \eqref{L1bds}, we obtain
	\begin{align*}
	\int_{\R^d} |\nabla \cgam(t)|^2  +\int_0^t\int_{\R^d} |\Delta \cgam|^2 \leq C(T) + \|\nabla \cgam^0\|^{2}_{L^2({\R^d})} ,
	\end{align*}
	for $0<t\leq T$, and thanks to \eqref{2 gradc} we conclude the proof of the first and second estimates in \eqref{2 c}.
	
	At last, considering the equation for $\cgam$
	\begin{equation*}
	\partial_t \cgam = \Delta \cgam -\n H(\cgam) + (c_B -\cgam)K(\p),
	\end{equation*}
	and using the previous bounds on $\n, \cgam$ and $\Delta \cgam$ we conclude that $\partial_t \cgam \in L^2(Q_T)$.
	%
	%
\paragraph*{$L^1$ bounds for the time derivatives of $\n$ and $\p$.}
	We differentiate the equation for $\n$ and we multiply it by $\sign{\{\partial_t \n\}}$
	\begin{equation}\label{dndt}
	\partial_t |\partial_t \n|-\gamma\Delta 
	(\n^{\gamma}|\partial_t \n|)\leq |\partial_t \n|G 
	+ \n \partial_p G |\partial_t \p|+
	\n \partial_c G \partial_t \cgam \sign\{\partial_t \n\}.
	\end{equation}
	We integrate in space using the monotonicity of $G$
	\begin{equation*}
	\frac{d}{dt}\|\partial_t \n(t)\|_{L^1(\R^d)}\leq \|G\|_{L^\infty(Q_T)}\|\partial_t \n(t)\|_{L^1(\R^d)}+\|\partial_c G\|_{L^\infty(Q_T)}\|\n(t)\|_{L^2(\R^d)}\|\partial_t\cgam(t)\|_{L^2(\R^d)}.
	\end{equation*}
	Thanks to \eqref{L1bds} and \eqref{2 c}, Gronwall's lemma gives
	\begin{equation*}	
	\|\partial_t \n(t)\|_{L^1(\R^d)}\leq C(T)  \|(\partial_t \n)^0\|_{L^1(\R^d)}\leq C(T),
	\end{equation*}
	where in the last inequality we used \eqref{timezero}.
	
	By integrating in $Q_t:=\R^d\times (0,t)$ the equation \eqref{dndt},  we obtain
	\begin{equation*}
	\|\partial_t \n(t)\|_{L^1(\R^d)}+\min|\partial_p G|\iint_{Q_t}\n |\partial_t \p| \leq C(T),	
	\end{equation*}
	thanks to \eqref{timezero} and the $L^1$ bounds proved above.
	Then, for the time derivative of the pressure it holds
	\begin{equation*}
	\|\partial_t \p\|_{L^1(Q_T)}\leq \iint_{Q_T \cap \{\n\leq 1/2\}}\gamma \n^{\gamma-1}|\partial_t \n| + 2 \iint_{Q_T \cap \{\n \geq 1/2\}} \n |\partial_t \p|\leq C(T).
	\end{equation*}
	We differentiate the equation for $\cgam$ and multiply it by $\sign{\{\partial_t \cgam\}}$
	\begin{equation*}
	\partial_t |\partial_t \cgam|-\Delta 
	(|\partial_t \cgam|)\leq -\partial_t \n H \sign\{\partial_t\cgam\} -\n H' |\partial_t\cgam|  -|\partial_t\cgam| K+(c_B-c)K' \partial_t \p \sign\{\partial_t\cgam\} . 
	\end{equation*}
	Integrating in space we obtain
	\begin{align*}
	\frac{d}{dt}\|\partial_t \cgam(t)\|_{L^1(\R^d)}\leq &\,\|H\|_{L^\infty(Q_T)}\|\partial_t \n(t)\|_{L^1(\R^d)}+ n_H \|H'\|_{L^\infty(Q_T)}\|\partial_t \cgam(t)\|_{L^1(\R^d)}\\&+c_B\|K'\|_{L^\infty(Q_T)}\|\partial_t p(t)\|_{L^1(\R^d)},
	\end{align*}
	and thanks to the previous bounds and Gronwall's lemma we have
	\begin{equation*}
	\|\partial_t \cgam(t)\|_{L^1(\R^d)}\leq C(T) \|(\partial_t \cgam)^0\|_{L^1(\R^d)}\leq C(T),
	\end{equation*}
	and this concludes the proof of \eqref{timeder}.
	%
	%
\paragraph*{$L^4$ bound for the gradient of $\cgam$.}
	Now, we prove that the gradient of $\cgam$ is bounded in $L^4$. Integration by parts gives
	\begin{align*}
	\int_{\R^d} |\nabla \cgam|^4 = - \int_{\R^d} \cgam \Delta \cgam |\nabla \cgam|^2 - \int_{\R^d} \cgam \nabla \cgam \cdot \nabla(|\nabla \cgam|^2).
	\end{align*}
	We use Young's inequality on the first term of the RHS and we get
	\begin{align*}
	\frac{1}{2} \int_{\R^d} |\nabla \cgam|^4 \leq \frac{1}{2} \int_{\R^d} \cgam^2 |\Delta \cgam|^2 - \int_{\R^d} \cgam \nabla \cgam \cdot \nabla(|\nabla \cgam|^2).
	\end{align*}
	We write the last term as
	\begin{align*}
	- \int_{\R^d} \cgam \nabla \cgam \cdot \nabla(|\nabla \cgam|^2) &= -2 \sum_{i,j} \int_{\R^d}\cgam\, \partial_i \cgam\, \partial_j \cgam\, \partial^2_{i,j} \cgam \\
	&\leq \frac{1}{4} \int_{\R^d} |\nabla \cgam|^4 + 4c_B^2\int_{\R^d}  \sum_{i,j}(\partial^2_{i,j}\cgam)^2
	\\
	&= \frac{1}{4} \int_{\R^d} |\nabla \cgam|^4 + 4c_B^2\int_{\R^d}   |\Delta \cgam|^2.  
	\end{align*}
	Thus, we have
	\[
	\frac{1}{4} \int_{\R^d} |\nabla \cgam|^4 \leq \left(\frac{1}{2}+4\right)  c_B^2\int_{\R^d}   |\Delta \cgam|^2 .
	\]
	and the $L^4$  estimate is proved.
	%
	%
\paragraph*{$L^2$ bound for the pressure gradient.}
	Since the pressure satisfies the equation~\eqref{pressure}, integrating  it in space we get
	\begin{align*}
	\frac{d}{dt}\int_{\R^d} \p (t)= -\gamma\int_{\R^d} |\nabla \p(t)|^2 + \gamma \int_{\R^d} \p (t)G(\p(t),\cgam(t)) +\int_{\R^d} |\nabla \p(t)|^2.
	\end{align*}
	Then, we integrate in time
	\begin{align*}
	(\gamma -1)\int_0^T\int_{\R^d} |\nabla \p|^2&=  \|\p(0)\|_{L^1(\R^d)}-\| \p(T)\|_{L^1(\R^d)} +\gamma \int_0^T\int_{\R^d} \p G(\p,\cgam),\\
	(\gamma -1)\int_0^T\int_{\R^d} |\nabla \p|^2&\leq C_0 + \gamma C(T),
	\end{align*}
	and this gives, since $\gamma >1$,
	$$
	\int_0^T\int_{\R^d} |\nabla \p|^2\leq  C(T).
	$$ 
	%
	%
\section{Stronger a priori estimates on $p_\gamma$}
	\label{Sec:est2}
	To establish the complementarity condition \eqref{crd} is equivalent to prove the strong compactness of $|\nabla p_\gamma |^2$. One step towards this goal is to prove compactness in space using the classical AB estimate, \cite{AB, CrPi1982}.  Here, major difficulties arise. As explained in the Introduction, since the reaction term can change sign the usual Aronson-B\'enilan lower bound cannot hold true, see \cite{PQV,PTV}. Moreover, we cannot apply the comparison principle because of the bad coupling in the system \eqref{sys}. Since the $L^\infty$ bound from below in the AB estimate is missing, we prove an $L^3$ version, adapting the method presented in~\cite{GPS}. Then, we show that the gradient of the pressure is bounded in $L^4(Q_T)$, which gives the compactness needed to pass to the limit.  
	
	Our first goal is to prove the AB estimate on the functional 
	\begin{equation}
	\label{w}
	w:= \Delta \p + G(\p,\cgam),
	\end{equation} 
	which is a variation of the Laplacian in order to take into account the source term, at the same order of $\Delta p_\gamma$, in equation~\eqref{pressure}. 
	
	\begin{thm}[Aronson-B\'enilan estimate in $L^3$]\label{Thm1}
		With the assumptions of Section~\ref{sec:NAMR} and with $\gamma >\max(1,2-\frac 4 d)$, for all  $T>0$ there is a constant $C(T)$   depending on $T$ and the previous bounds and independent of $\gamma$  such that  
		\begin{equation}\label{AB}
		\int_0^T \int_{\Omega_T} |w|_-^3 \leq C(T),\quad
		\int_0^T \int_{\R^d} |\Delta \p| \leq C(T).
		\end{equation}
	\end{thm}
	%
	Let us point out that because the free boundary is where $p_\infty$ vanishes, it is important that $w$ itself is controlled and not merely $p w$ as in the next estimate.
	\begin{thm}[$L^4$ estimate on the pressure gradient]\label{L4}
		With the same assumptions as before, given $T>0$, it holds
		\begin{equation}\label{t4}
		(\gamma -1)    \int_0^T\int_{\Omega_T} \p |\Delta \p+G|^2 + \int_0^T\int_{\Omega_T} \p \sum_{i,j} (\partial^2_{i,j}\p)^2 \leq C(T),
		\end{equation}
		\begin{equation}\label{grad p 4}
		\int_0^T\int_{\Omega_T} |\nabla \p|^4 \leq C(T),
		\end{equation}
		where $C$ depends on $T$ and previous bounds and is independent of $\gamma$.
	\end{thm}
	We recall that in the model independent of $c$, \cite{PQV}, the AB estimate is much stronger and gives $\Delta \p(t) + G(\p(t)) \geq -\frac{1}{\gamma t}$, and the major difficulty is the control of $\Delta \p$ which is provided by Theorem~\ref{AB}. As proved in \cite{MePeQu}, the $L^4$ estimate follows from the total energy control when $G=G(p)$, but this uses the strong form of  the AB estimate. Therefore, we have invented another proof, which is reminiscent of the energy control, but uses a different treatment of the 'dissipation' terms.
	%
	%
\paragraph*{Proof of Theorem~\ref{Thm1}.}
	For the sake of simplicity we forget the index $\gamma$ in this proof. We compute the time derivative of $w$ and obtain
	\begin{equation*}
	\partial_t w =\Delta(|\nabla p|^2)+\gamma \Delta(p w)+ \partial_p G (|\nabla p|^2+\gamma p w)+\partial_c G\, \partial_t c.
	\end{equation*}
	The first term is 
	\begin{equation*}
	\Delta (|\nabla p|^2)= 2\sum_{i,j} (\partial^2_{i,j}p)^2  + 2\nabla p \cdot\nabla(\Delta p) \geq \frac{2}{d}(\Delta p)^2 + 2\nabla p \cdot\nabla(\Delta p).
	\end{equation*}
	By definition of $w$ we have      
	\begin{align*}
	2\nabla p \cdot \nabla(\Delta p)=2\nabla p \cdot \nabla (w-G)=2\nabla p\cdot\nabla w -2\partial_p G |\nabla p|^2 -2 \partial_c G\nabla p \cdot \nabla c.
	\end{align*} 
	Hence, the time derivative satisfies
	\begin{align}\label{w_}
	\partial_t w\geq &\frac{2}{d}(w-G)^2+2\nabla p\cdot\nabla w -\partial_p G |\nabla p|^2 -2 \partial_c G\nabla p \cdot \nabla c\\\nonumber
	&+\gamma \Delta(p w)
	+\gamma p w \,\partial_p G +\partial_c G \,\partial_t c.
	\end{align}
	Multiplying \eqref{w_} by $-|w|_-$, we obtain
	\begin{align*}
	-\partial_t w\, |w|_- \leq &-\frac{2}{d}|w|_-^3-\frac{4}{d} G|w|^2_--\frac{2}{d}G^2|w|_- +  \nabla p\cdot \nabla |w|_-^2 +\partial_p G |\nabla p|^2|w|_- \\
	&+2 \partial_c G\nabla p \cdot \nabla c |w|_-
	+\gamma \Delta(p |w|_-)|w|_-+\gamma p\, \partial_p G  |w|_-^2 - \partial_c G\, \partial_t c |w|_-.
	\end{align*}
	Hence, using the fact that $\partial_p G< -\beta$ from \eqref{ReTe}, we integrate in space to obtain
	\begin{align*}
	\frac{d}{\dt}\int_{\Omega_T}\frac{|w|_-^2}{2}\leq &-\frac{2}{d}\int_{\Omega_T}|w|_-^3 -\frac{2}{d}\int_{\Omega_T}   G^2 |w|_- -\beta \int_{\Omega_T} |\nabla p|^2|w|_- \\&-\frac{4}{d} \int_{\Omega_T}   G |w|_-^2 
	+\underbrace{\int_{\Omega_T}\left[ \nabla p \cdot \nabla |w|_-^2 + \gamma \Delta(p |w|_-)|w|_- \right] }_{\A} 
	\\&\underbrace{-\int_{\Omega_T} \partial_c G\,\partial_t c |w|_- }_{\B} +\underbrace{2\int_{\Omega_T} \partial_c G\nabla p \cdot \nabla c |w|_-}_{\C},
	\end{align*}
	where $C$ is a positive constant depending on  $\|G\|_\infty$ and $d$. 
	Now we proceed integrating by parts each term.
	\begin{align*}
	\nonumber  \A&=-\int_{\Omega_T}\left[ \Delta p |w|_-^2  + \gamma  \nabla p \nabla |w|_- |w|_- + \gamma p |\nabla|w|_-|^2\right]
	\\
	\nonumber
	&= \int_{\Omega_T} |w|_-^3  +\int_{\Omega_T} G |w|_-^2+ \frac{\gamma}{2}\int_{\Omega_T} \Delta p |w|_-^2-\gamma \int_{\Omega_T}p |\nabla|w|_-|^2\\
	\label{rel1}
	&=\left(1-\frac{\gamma}{2}\right) \int_{\Omega_T}|w|_-^3 +\left(1-\frac{\gamma}{2}\right) \int_{\Omega_T}G|w|_-^2-\gamma\int_{\Omega_T}p|\nabla |w|_-|^2.
	\end{align*}            
	Next, using \eqref{2 c} and the Cauchy-Schwarz inequality, we obtain 
	\begin{equation*}
	\B\leq\, C \int_{\Omega_T}|w|_-^2+C.
	\end{equation*}        
Thanks to Young's inequality and \eqref{c4p2}, we compute 
	\begin{align*}
	\C&\leq \frac \beta 2 \int_{\Omega_T}|\nabla p|^2  |w|_- + C \int_{\Omega_T}| \nabla c |^4 + C \int_{\Omega_T}| w|_-^2 \\
	&\leq \frac \beta 2 \int_{\Omega_T}|\nabla p|^2  |w|_- + C \int_{\Omega_T}| w|_-^2 +C.
	\end{align*} 
	We may now come back to the control of $\frac{d}{dt}\int_{\Omega_T}\frac{|w|_-^2}{2}$.
	Gathering all the previous bounds, we get the following estimate
	\begin{align*}
	\frac{d}{\dt}\int_{\Omega_T}\frac{|w|_-^2}{2}&\leq  -\left(\frac{2}{d}-1+\frac{\gamma}{2}\right)\int_{\Omega_T}|w|_-^3 - \frac{\beta}{2} \int_{\Omega_T}|\nabla p|^2|w|_- +C(\gamma+1)\int_{\Omega_T} |w|_-^2 +C.
	\end{align*}
	Hence integrating in time we have
	\begin{align*}\label{time}
	\left(\frac{2}{d}-1+\frac{\gamma}{2}\right) \int_0^T\int_{\Omega_T}|w|_-^3
	&\leq   C\, (\gamma +1) \int_0^T\int_{\Omega_T} |w|_-^2 + \int_{\Omega_T} \frac{|w^0|_-^2}{2} +C(T)\\\nonumber
	&\leq  C\, (\gamma +1) \left(\int_0^T\int_{\Omega_T}|w|_-^3\right)^\frac{2}{3} +C(T),
	\end{align*}
	where we used the assumption \eqref{2 deltap} and C represents different constants depending on $T$, $\left|\Omega(T) \right|$ and previous bounds. This is the place where we strongly use the compact support assumption. 
	
	At last, with our assumption that $\gamma$ is large enough, we obtain
	\begin{align*}
	\int_0^T \int_{\Omega_T}|w|_-^3&\leq  C\left(\int_0^T\int_{\Omega_T}|w|_-^3\right)^\frac{2}{3}+C(T),
	\end{align*}   
	and hence we have proved our main result, that is the first estimate of \eqref{AB}, 
	\begin{equation*}
	\int_0^T\int_{\Omega_T}|w|_-^3\leq  C(T).
	\end{equation*}
	To prove the second estimate, we argue as follows. Since
	\begin{equation*}
	\int_0^T \int_{\Omega_T} (\Delta p + G)\leq C(T),
	\end{equation*}
	we can also control the positive part of $w$
	\begin{equation*}
	\int_0^T \int_{\Omega_T} |w|_+\leq C(T) +  \int_0^T \int_{\Omega_T} |w|_- \leq C(T)+ C \left( \int_0^T \int_{\Omega_T} |w|_-^3\right)^{\frac 13}.
	\end{equation*}
	Thus it holds
	\begin{equation*}
	\int_0^T \int_{\Omega_T} |\Delta p + G| \leq C(T).
	\end{equation*}
	Hence, we finally obtain the $L^1$ estimate for the Laplacian of the pressure
	\begin{equation*}
	\int_0^T \int_{\Omega_T} |\Delta p| \leq C(T),
	\end{equation*}
	that concludes the proof of Theorem~\ref{Thm1}.
	%
	%
\paragraph*{Proof of Theorem~\ref{L4}.}
	We consider the equation for the pressure \eqref{pressure}, we multiply it by $-(\Delta \p + G(\p,\cgam))$ and integrate in space. We find successively
	\begin{equation*}
	-\int_{\Omega_T} \partial_t \p \Delta \p - \int_{\Omega_T} \partial_t \p\, G = -\gamma \int_{\Omega_T} \p |\Delta \p + G|^2 - \int_{\Omega_T} |\nabla \p|^2 \Delta \p - \int_{\Omega_T} |\nabla \p|^2 G ,
	\end{equation*}
	\begin{equation*}
	\frac{d}{\dt} \int_{\Omega_T} \frac{|\nabla \p|^2}{2} -\int_{\Omega_T} \partial_t \p G+\gamma \int_{\Omega_T} \p |\Delta \p + G |^2 +\int_{\Omega_T} |\nabla \p|^2 \Delta \p \leq \|G\|_{L^\infty} \|\nabla \p (t)\|^2_{L^2}.
	\end{equation*}
	We integrate by parts the last term of the LHS and obtain
	\begin{align*}
	\int_{\Omega_T} |\nabla \p|^2 \Delta \p &=\int_{\Omega_T} \p \Delta(|\nabla \p|^2)\\
	& = 2 \int_{\Omega_T} \p \nabla \p \cdot \nabla (\Delta \p) +2 \int_{\Omega_T}\p \sum_{i,j} (\partial_{i,j}^2\p)^2\\
	&= -2 \int_{\Omega_T}\p |\Delta \p|^2 - 2\int_{\Omega_T} |\nabla \p|^2\Delta \p  +2 \int_{\Omega_T}\p \sum_{i,j} (\partial^2_{i,j}\p)^2.
	\end{align*}
	Hence, we conclude that
	\begin{equation*}
	\int_{\Omega_T} |\nabla \p|^2 \Delta \p= - \frac{2}{3} \int_{\Omega_T}\p |\Delta \p|^2 +\frac{2}{3} \int_{\Omega_T}\p \sum_{i,j} (\partial^2_{i,j}\p)^2.
	\end{equation*}
	Thus, we have
	\begin{equation}\label{diseq}
	\frac{d}{\dt} \int_{\Omega_T} \frac{|\nabla \p|^2}{2} \underbrace{-\int_{\Omega_T} \partial_t \p\, G}_{I_1}+\underbrace{\gamma \int_{\Omega_T} \p |\Delta \p + G|^2 -\frac{2}{3}\int_{\Omega_T} \p |\Delta \p|^2}_{I_2}  +\frac{2}{3} \int_{\Omega_T}\p \sum_{i,j} (\partial^2_{i,j}\p)^2 \leq C(T).
	\end{equation}
	We can define the function $\overline{G}=\overline{G}(\p,\cgam)=\int_0^{\p} G(q,\cgam) dq$ and then
	\begin{equation*}
	\partial_t \p \,G(\p,\cgam)=\partial_t \overline{G}(\p,\cgam)-\partial_t \cgam\, \partial_c \overline{G}(\p,\cgam).
	\end{equation*}
	Using this relation the term $I_1$ can be written as
	\begin{equation*}
	I_1=-\int_{\Omega_T}\partial_t \overline{G}+\int_{\Omega_T}\partial_c\overline{G}\, \partial_t \cgam \geq -\int_{\Omega_T}\partial_t \overline{G}-  C
	\end{equation*}
	thanks to the $L^2$ bound on $\partial_t \cgam$ in~\eqref{2 c} and because  $|\partial_c\overline{G} | \leq C \p$.
	We can estimate the term $I_2$ from below as follows
	\begin{equation*}
	I_2\geq\left(\gamma - 1 \right)\int_{\Omega_T}\p |\Delta \p+G|^2- C \int_{\Omega_T}\p |G|^2.
	\end{equation*}
	Therefore
	\begin{equation}\label{I1I2}
	I_1+I_2\geq \left(\gamma - 1  \right)\int_{\Omega_T} \p |\Delta \p +G|^2  -\int_{\Omega_T}\partial_t \overline{G} -C(T).
	\end{equation}
	Combining \eqref{diseq} and \eqref{I1I2}, we obtain
	\begin{equation*}
	\frac{d}{dt}\int_{\Omega_T} \left[ \frac{|\nabla \p|^2}{2} - \overline{G} \right] +(\gamma -1) \int_{\Omega_T}\p|\Delta \p+G|^2 + \frac{2}{3}\int_{\Omega_T} \p \sum_{i,j} (\partial^2_{i,j}\p)^2 \leq C(T).
	\end{equation*}
	Finally, integrating in time, we obtain the estimate~\eqref{t4} and this proves the first step of Theorem~\ref{L4}. 
	
	Furthermore, this bound also implies
	\begin{equation}\label{pDelta}
	(\gamma -1) \int_0^T\int_{\Omega_T} \p |\Delta \p|^2 \leq C(T) .
	\end{equation}
	Now, we compute the $L^4$ norm of the gradient of $\p$, as we did for the gradient of $\cgam$.
	\begin{align*}
	\int_{\Omega_T} |\nabla \p|^4 = - \int_{\Omega_T} \p \Delta \p |\nabla \p|^2 - \int_{\Omega_T} \p \nabla \p\cdot \nabla(|\nabla \p|^2).
	\end{align*}
	Applying Young's inequality to the first term, we obtain
	\begin{align*}
	\frac{1}{2} \int_{\Omega_T} |\nabla \p|^4 \leq \frac{1}{2} \int_{\Omega_T} \p^2 |\Delta \p|^2 - 2 \sum_{i,j} \int_{\Omega_T}\p\, \partial_i \p\, \partial_j \p \, \partial^2_{i,j} \p .
	\end{align*}
	The last term can be upper bounded by 
	\begin{align*}
	 \frac{1}{4} \int_{\Omega_T} |\nabla \p|^4 + 4\int_{\Omega_T} \p^2 \sum_{i,j}(\partial^2_{i,j}\p)^2.  
	\end{align*}
	Therefore,  we obtain
	\begin{equation*}
	\frac{1}{4} \int_{\Omega_T} |\nabla \p|^4 \leq \frac{1}{2} \int_{\Omega_T} \p^2 |\Delta \p|^2  + 4\int_{\Omega_T} \p^2 \sum_{i,j}(\partial^2_{i,j}\p)^2.
	\end{equation*}
	Since $p\leq p_H$, by~\eqref{t4} and \eqref{pDelta} we conclude
	\begin{align*}
	\int_0^T\int_{\Omega_T} |\nabla \p|^4 \leq& C(T),
	\end{align*}
	and this completes the proof of Theorem~\ref{L4}.

\section{Complementarity relation}
	\label{Sec:4}
	
	Thanks to the bounds provided by Theorem~\ref{Thm1} and Theorem~\ref{L4}, we may obtain the desired compactness on the pressure gradient. This allows us to pass to the incompressible limit and prove the complementarity relation as we state it now.
	\begin{thm}[Complementarity relation]\label{CR}
		With the assumptions of Theorem~\ref{Thm1}, the complementarity condition~\eqref{crd} holds. More precisely, for all test functions $\zeta \in \D(Q)$, the limit pressure $p_\infty$ satisfies 
		\begin{equation}\label{compl}
		\iint_{Q}  \left( -|\nabla p_\infty|^2\zeta - p_\infty\nabla p_\infty\cdot\nabla\zeta + p_\infty G(p_\infty,c_\infty)\zeta\right)=0.
		\end{equation}
	\end{thm}
	This result is related to the geometric form of the Hele-Shaw free boundary problem (while~\eqref{syslim} is the weak form). It tells us that the limit solution satisfies 
	\begin{equation*}
	\begin{cases}
	-\Delta p_\infty =  G(p_\infty, c_\infty) &\text{ in } {\cal O}(t) :=\{x;\, p_{\infty}(x,t) >0\} ,\\
	p_\infty=0 &\text{ on } \partial {\cal O}(t),
	\end{cases}
	\end{equation*}
	where, for every $t>0$, the set ${\cal O}(t)$ represents the region occupied by the tumor. Moreover, in the limit, the pressure and the cell population density satisfy the relation
	\begin{equation}
	\label{relpn}
	p_\infty(1-n_\infty)=0.
	\end{equation}
	In fact, we may expect that the set ${\cal O}(t)$ coincides a.e. with the set where $n_\infty=1$, hence the classification of \textit{incompressible} model. See \cite{MePeQu} for the proof in the case without nutrient. It is not obvious to extend the result in the case at hand.
	%
	%
	%
\paragraph*{Proof of Theorem~\ref{CR}.}
	Thanks to the bounds in~\eqref{2 c}, \eqref{timeder} and \eqref{c4p2}, $\p$ and $\cgam$ are locally compact and thus, after the extraction of sub-sequences,
	\begin{align*}
	\p \rightarrow p_\infty \text{ strongly in } L^1(Q_T),\qquad 
	\cgam \rightarrow c_\infty \text{ strongly in } L^1(Q_T),
	\end{align*}
	when $\gamma \rightarrow \infty$, for all $T>0$.
	From Theorem~\ref{L4}, we also recover the weak convergence of the gradient of the pressure, up to a sub-sequence,
	\begin{equation*}
	\nabla \p \rightharpoonup \nabla p_\infty  \text{ weakly in } L^4(Q_T).
	\end{equation*}
	From Theorem~\ref{Thm1}, we know that $\Delta \p$ is bounded in $L^1$. Then, we have local compactness in space for the pressure gradient. To gain compactness in time we use the Aubin-Lions lemma. From the equation for the pressure \eqref{pressure}, we have
	\begin{equation*}
	\partial_t (\nabla\p)=\nabla[\gamma \p(\Delta  \p +G) + |\nabla \p|^2],
	\end{equation*}
	where the RHS is a sum of space derivatives of functions bounded in $L^1$. In fact, since by \eqref{timeder} and \eqref{c4p2}, $\partial_t \p$ and $|\nabla \p|^2$ are in $L^1$, from \eqref{pressure} the term $\gamma \p (\Delta\p +G)$ is also bounded in $L^1$.
	Thus, we can extract a sub-sequence such that
	\begin{equation*}
	\nabla \p \rightarrow \nabla p_\infty \text{ strongly in } L^q(Q_T), \text{ for } 1 \leq q < \frac{d}{d-1}.
	\end{equation*}
	After extraction of a sub-sequence we obtain convergence almost everywhere for $\nabla \p$. Then, using the $L^4$ bound of Theorem~\ref{L4}, we have
	\begin{equation*}
	\nabla \p \rightarrow \nabla p_\infty \text{ strongly in } L^q(Q_T), \text{ for } 1 \leq q < 4,
	\end{equation*}
	hence, in particular, also for $q=2$.
	
	Let $\zeta \in \D(Q)$ be a test function. We consider the equation for $\p$
	\begin{equation*}
	\partial_t \p =\gamma \p (\Delta \p + G(\p, \cgam)) + |\nabla \p|^2,
	\end{equation*}
	we multiply it by $\zeta$ and we integrate in $Q$
	\begin{equation}\label{relation}
	- \frac{1}{\gamma} \iint_{Q} \left( \p  \partial_t\zeta+|\nabla \p|^2\zeta \right)=\iint_{Q}  \left( -|\nabla \p|^2\zeta - \p\nabla \p\cdot\nabla\zeta + \p G(\p,\cgam)\zeta\right)
	\end{equation}
	Hence, passing to the limit for $\gamma \rightarrow \infty$ we obtain the complementarity relation
	\begin{equation}
	\iint_{Q}  \left( -|\nabla p_\infty|^2\zeta - p_\infty\nabla p_\infty\cdot\nabla\zeta + p_\infty G(p_\infty,c_\infty)\zeta\right)=0.
	\end{equation}
	This is equivalent to
	\begin{equation*}
	\iint_{Q}  p_\infty \left( \Delta p_\infty + G(p_\infty,c_\infty)\right) \zeta=0,
	\end{equation*}
	which means
	\begin{equation*}
	p_\infty \left( \Delta p_\infty + G(p_\infty,c_\infty)\right) =0, \;\text{ in } \D'(Q),
	\end{equation*}
	and the proof of Theorem~\ref{CR} is complete.
\section*{Acknowledgments}

\includegraphics[width=0.7cm]{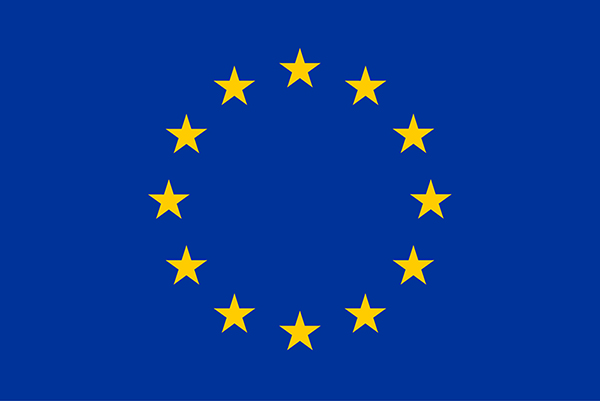} This project has received funding from the European Union’s Horizon 2020 research and innovation program under the Marie Skłodowska-Curie (grant agreement No 754362).\\
B.P. has received funding from the European Research Council (ERC) under the European Union's Horizon 2020 research and innovation program (grant agreement No 740623).

	%
	%
\begin{appendices}
	%
	%
\section{Compact support property}
		\label{appA}
		We now give the proof of the finite speed of propagation property of the solutions of system \eqref{sys}. Our goal is to show that if the initial data satisfy
		$$\text{supp}(n_\gamma^0)\subset \Omega_0, \quad \forall \gamma >1,$$
		with $\Omega_0$ independent of $\gamma$, then the solutions $\n(t)$, $\p(t)$ are compactly supported, uniformly in $\gamma$ and $t\in [0,T]$, for all $T>0$. This means that there exists a bounded open domain $\Omega_T$ independent of $\gamma$ such that
		$$\text{supp}(n_\gamma(t))\subset \Omega_T, \quad \forall\gamma>1,\, \forall t\in [0,T].$$
		For every $\gamma>1$, the pressure $\p$ is a sub-solution to the equation
		\begin{equation}
		\partial_t \p \leq |\nabla \p|^2+ \gamma \p(\Delta \p + G(0,c_B)),	
		\end{equation}
		therefore, finding a super-solution with compact support we can control the supports of $\p$ and $\n$.
		
		We consider the function
		$$\Pi (x,t) = G(0,c_B)\left|S(t)-\frac{|x|^2}{2}\right|_+,$$
		where we choose the function $S$ such that it satisfies $$S'(t)\geq 2 G(0,c_B) S(t).$$
		We compute the derivatives of $\Pi$ and we find
		\begin{align*}
		&\partial_t \Pi (x,t)=G(0,c_B)S'(t)\mathds{1}_{\{S(t)\geq \frac{|x|^2}{2}\}   }   ,\\
		&\nabla \Pi(x,t)=-G(0,c_B) x \mathds{1}_{\{S(t)\geq \frac{|x|^2}{2}\}   }  ,\quad
		\Delta \Pi (x,t)\leq-d\, G(0,c_B)\mathds{1}_{\{S(t)\geq \frac{|x|^2}{2}\} }   .
		\end{align*}
		Therefore $\Pi$ satisfies
		\begin{align*}
		\partial_t \Pi -|\nabla \Pi|^2-\gamma \Pi(\Delta \Pi +G(0,c_B))&\geq
		(  G(0,c_B)S'(t)  -G(0,c_B)^2 x^2)\mathds{1}_{\{S(t)\geq \frac{|x|^2}{2}\}}+\gamma \Pi G(0,c_B)(d+1)\\
		&\geq ( 2 G(0,c_B)^2 S(t)  -G(0,c_B)^2 x^2)\mathds{1}_{\{S(t)\geq \frac{|x|^2}{2}\} }\\
		&\geq 0.
		\end{align*}
		Hence, we have proved that for all $T>0$ 
		$$\text{supp}(\p(t))\subset \text{supp}(\Pi(t))\subset B_{T},\;\forall \gamma >1, \forall t\in[0,T],$$
		where $B_{T}$ is the open ball with radius $\sqrt{2 S(T)}$.
		%
		%
\section{Removing the compact support assumption}
		\label{app}
		The proof of the main result of the paper is built on the compact support assumption stated in Section~\ref{sec:NAMR}. Our goal is to generalize the result removing this condition. Let us note that it is sufficient to extend the Theorem~\ref{Thm1}, since it is the only one for which we used the compact support assumption. Moreover, let us notice that Proposition~\ref{prop1} holds true in this framework.
		We define the functional $w$ as in \eqref{w} and we state the following result.
		\begin{prop}[Aronson-B\'enilan generalized estimate in $L^3$]
			Let $\Phi$ be a test function in $C_{\text{comp}}^2(\R^d)$. With the assumptions from \eqref{ReTe} to \eqref{2 gradc}, and with $\gamma > \max(1,2-\frac 4 d)$, for all $T>0$ there exists a constant $C(T)$ depending on the previous bounds and independent of $\gamma$ such that
			\begin{equation}
			\int_0^T\int_{\R^d}|w|_-^3\Phi\leq C(T), \quad  \int_0^T \int_{\R^d} |\Delta \p| \Phi \leq C(T).
			\end{equation}
		\end{prop} 
		\begin{proof}
			Computing the time derivative of the negative part of $w$, we have
			\begin{align*}
			-\partial_t \left(\frac{|w|_-^2}{2}\right) \leq    
			&-\frac{4}{d}|w|_-^3-\frac{2}{d} G|w|^2_--\frac{2}{d}G^2|w|_-+\nabla |w|_-^2 \cdot \nabla p+\partial_p G |\nabla p|^2|w|_-\\ &+2 \partial_c G\nabla p \cdot \nabla c |w|_-
			+\gamma \Delta(p |w|_-)|w|_- - \partial_c G\, \partial_t c |w|_-.
			\end{align*}
			as in the proof of Theorem~\ref{Thm1}.
			We multiply the inequality by $\Phi$ and integrate in space
			\begin{align}\label{phi} 
			\frac{d}{\dt}\int_{\Omega_T}\frac{|w|_-^2}{2}\Phi\leq &-\frac{2}{d}\int_{\Omega_T}|w|_-^3\Phi -\frac{2}{d}\int_{\Omega_T}   G^2 |w|_-\Phi-\beta\int_{\Omega_T} |\nabla p|^2|w|_- \Phi
			\\\nonumber & -\frac{4}{d} \int_{\Omega_T}   G |w|_-^2 \Phi +\underbrace{\int_{\Omega_T}\left[\nabla p\cdot \nabla \left(|w|_-^2\right) \Phi + \gamma \Delta(p |w|_-)|w|_-\Phi \right] }_{\A} \\
			\nonumber&\underbrace{-\int_{\Omega_T} \partial_c G\,\partial_t c |w|_- \Phi}_{\B} +\underbrace{2\int_{\Omega_T} \partial_c G\nabla p \cdot \nabla c |w|_-\Phi}_{\C}.
			\end{align}
			Now we proceed computing each term.
			\begin{align*}
			\A=&\int_{\R^d}\nabla p \cdot\nabla \left(|w|_-^2\right) \Phi -\gamma \int_{\R^d} \nabla (p |w|_-)\cdot\nabla |w|_- \Phi -\gamma \int_{\R^d}|w|_- \nabla (p |w|_-)\cdot\nabla \Phi \\
			= &-\int_{\R^d}\Delta p |w|_-^2 \Phi - \int_{\R^d}|w|_-^2 \nabla p \nabla\Phi -\gamma \int_{\R^d}|w|_-\nabla p\nabla|w|_-\Phi \\
			&-\gamma \int_{\R^d}p |\nabla |w|_-|^2\Phi+\gamma \int_{\R^d}p|w|_-^2\Delta\Phi+\gamma \int_{\R^d} p \nabla \left(\frac{|w|_-^2}{2}\right)\cdot\nabla \Phi\\
			= &-\int_{\R^d}\Delta p |w|_-^2 \Phi - \int_{\R^d}|w|_-^2 \nabla p \cdot\nabla\Phi +\frac\gamma 2 \int_{\R^d}\Delta p|w|_-^2\Phi +\frac\gamma 2 \int_{\R^d}|w|_-^2\nabla p\cdot\nabla\Phi \\
			&-\gamma \int_{\R^d}p |\nabla |w|_-|^2\Phi +\frac{\gamma}{2} \int_{\R^d} p |w|_-^2 \Delta \Phi-\frac \gamma 2 \int_{\R^d}|w|_-^2 \nabla p \cdot\nabla \Phi\\
			=&\left(1-\frac\gamma 2\right)\int_{\R^d} |w|_-^3\Phi  +\left(1- \frac \gamma 2\right)\int_{\R^d} G |w|_-^2\Phi -\gamma \int_{\R^d} p |\nabla|w|_-|^2\Phi+\A_1,
			\end{align*}
			with $$\A_1=\frac \gamma 2 \int_{\R^d}p|w|_-^2\Delta\Phi -\int_{\R^d}|w|_-^2\nabla p\cdot\nabla\Phi.$$
			By the Cauchy-Schwarz inequality we have
			\begin{align*}
			\B&\leq\,  \int_{\R^d}|w|_-^2\Phi+C\int_{\R^d}|\partial_t c|^2\Phi\leq \int_{\R^d}|w|_-^2\Phi+C.
			\end{align*}
			Using Young's inequality and \eqref{c4p2}, we find 
			\begin{align*}  
			\C&\leq \frac{\beta}{2}\int_{\R^d} | \nabla p|^2|w|_- \Phi+ C\int_{\R^d} |\nabla c|^2|w|_-\Phi\\
			&\leq  \frac{\beta}{2}\int_{\R^d} |\nabla p|^2|w|_- \Phi+ C\int_{\R^d} |\nabla c|^4\Phi+ C\int_{\R^d}|w|_-^2\Phi \\
			&\leq  \frac{\beta}{2}\int_{\R^d} |\nabla p|^2|w|_- \Phi+ C\int_{\R^d}|w|_-^2\Phi +C.
			\end{align*}
			It remains to treat the term containing the derivatives of $\Phi$
			\begin{equation*}
			\A_1= -\int_{\R^d}|w|_-^2\nabla p\cdot \nabla\Phi+\frac \gamma 2 \int_{\R^d}p|w|_-^2\Delta\Phi.
			\end{equation*}
			We choose a positive function $\Phi$ with exponential decay, such that $|\nabla\Phi|\leq C\Phi$ and  $|\Delta\Phi|\leq C\Phi$. Now, we integrate by parts and use Young's inequality 
			\begin{align*} 
			\A_1=&\,2\int_{\R^d}p|w|_-\nabla|w|_- \cdot\nabla\Phi +\left(1+\frac{\gamma}{2}\right)\int_{\R^d}p|w|_-^2  \Delta \Phi \\
			\leq & \,\frac{1}{2} \int_{\R^d}p|\nabla|w|_-|^2\Phi +  C (\gamma+1) \int_{\R^d}|w|_-^2\Phi.
			\end{align*} 
			Finally, the inequality \eqref{phi} can be written as follows
			\begin{equation*}
			\frac{d}{\dt}\int_{\R^d}|w|_-^2\Phi+ \left(\frac 2 d+\frac \gamma 2 -1\right)\int_{\R^d}|w|_-^3\Phi+\frac \beta 2\int_{\R^d}|\nabla p|^2 |w|_-\Phi\leq C(\gamma+1)\int_{\R^d} |w|_-^2\Phi  +C,
			\end{equation*}
			then, for $ \gamma>2-\frac 4 d$, integrating in time we have
			\begin{equation*}
			\int_0^T\int_{\R^d}|w|_-^3\Phi\leq \left(\int_0^T\int_{\R^d}|w|_-^3\Phi\right)^\frac{2}{3}+C(T),
			\end{equation*}
			and then we have proved
			\begin{equation*}
			\int_0^T\int_{\R^d}|w|_-^3\Phi\leq  C(T) .
			\end{equation*}
			By consequence
			\begin{equation*}
			\int_0^T\int_{\R^d}|w|_-^2\Phi\leq  C(T),\quad\int_0^T \int_{\R^d}|w|_-\Phi\leq  C(T). 
			\end{equation*}
			Since $\Phi$ is a smooth function with compact support
			\begin{equation*}
			\int_0^T \int_{\R^d} (\Delta p + G)\Phi\leq C,
			\end{equation*}
			and then also 
			$$\int_{\R^d}\Phi|\Delta p + G|_+=\int_{\R^d}\Phi(\Delta p + G)+\int_{\R^d}\Phi|\Delta p + G|_-\leq C(T).$$
			Therefore we recover the local $L^1$ estimate for the Laplacian of the pressure
			\begin{equation*}
			\int_0^T \int_{\R^d} |\Delta p|\Phi \leq C.
			\end{equation*}
		\end{proof}
%
%
\section{Optimality of  the bound $\nabla p \in L^4$}

In Theorem \ref{L4}, we have established the uniform bound $\nabla p \in L^4$, see \eqref{grad p 4}. Here, we aim to show that the exponent $4$ cannot be increased. We use the  so-called \textit{focusing solution} from~\cite{AG} that we adapt to the limit $\gamma \to \infty$, i.e., the Hele-Shaw problem. We recall that, for the porous medium equation,  the focusing solution consists in a spherical hole filling which generates a  stronger singularity than the Barenblatt solution, see \cite{V}. 

Consider $\alpha>0$ such that $\nabla p \in L^\alpha(Q_T)$, where $p$ is a solution of the Hele-Shaw problem with Dirichlet boundary conditions in a spherical shell $\{R(t) < | x| < R_1\}$,  for a given $R_1>0$ and $R(0)$ small enough. Then, to simplify the problem, we fix the external radius $R_1$ and let $p$~satisfy
\begin{equation} \label{p'} 
\begin{cases}
-\Delta p = 1,\quad  &\text{ for  }  \quad R(t) < | x| < R_1,
\\[2pt]
p(x)=0, \quad  &\text{ for } \quad |x|= R(t) \ \text{ or } |x|= R_1 ,
\\[2pt]
R'(t)=-\nabla p \cdot \vec{n},  &\text{ for } \quad |x|= R(t).
\end{cases}
\end{equation}
Here, $\vec{n}$ denotes the inner normal to the ball $B_{R(t)}(0)$. As in~\cite{AG},   $R(t) $ diminishes and  vanishes in finite time,  generating a singularity $|\nabla p| \to \infty$. The power $4$ turns out to be the highest possible integrability in time at this singular time.
We treat the case of dimension 2. In higher dimension the radial solutions are more regular and the worse singularity would be obatined for a cylinder with a 2 dimensional basis.

\paragraph*{Case $d=2$.}
With spherical symmetry, we set $p:=p(r)$, $r=|x|$ and equation \eqref{p'} reads 
$$
\frac 1 r (rp')' =1.
$$
Integrating once, we get, for some $a(t)$
\begin{equation*}
p'=-\frac{r}{2}+\frac{a(t)}{r},
\end{equation*}
and the second integration yields
\begin{equation*}
p=-\frac{r^2}{4}+a(t) \ln r +b(t).
\end{equation*}
Imposing $p(R_1)=p(R(t))=0,$ we find
\begin{align*}
b(t)&=\frac{R_1^2}{4}-a(t) \ln R_1,\\
\frac{R(t)^2}{4}- a(t) \ln R(t) &= \frac{R_1^2}{4}-a(t) \ln R_1.
\end{align*}
Hence for $R(t)\approx 0$, we have
\begin{equation}\label{a}
a(t) \approx - \frac{R_1^2}{4 \ln R(t)}, \qquad \qquad R'(t)\approx\frac{1}{ R(t) \; \ln R(t) } .
\end{equation}
Therefore there is  $T>0$ when $R(T^-)=0$ and as $t \approx T$, we compute
\begin{equation*}
\int_0^T \int_{B_{R_1}(0)} |\nabla p(x)|^\alpha dx dt = \int_0^T \int_{R(t)}^{R_1} |p'(r)|^\alpha r   dr dt\approx\int_0^T \int_{R(t)}^{R_1} \frac{|a(t)|^\alpha}{r^{\alpha -1}}   dr dt.
\end{equation*}
The singularity at $T$ is thus driven by 
\begin{equation*}
\int_0^T\frac{|a(t)|^\alpha}{R(t)^{\alpha -2}} dt\approx  \int_0^T\frac{1}{| \ln R(t) |^{\alpha}\, R(t)^{\alpha-2}} dt \approx 
\int_{0} ^{R(0)} \frac{1}{|\ln R|^{\alpha-1} R^{\alpha-3}} dR
\end{equation*}
by the change of variable $R=R(t)$ and using  equation \eqref{a}. We recall that we have chosen $R(0)$ small enough.

This integral is finite for  $1\leq \alpha \leq 4$ and infinite for $\alpha >4$.
%
%
\end{appendices}
\bibliographystyle{plain}
\bibliography{biblio}

\begin{thebibliography}{10}

\bibitem{AB}
{Aronson, D. G. and B\'enilan, P.}
\newblock R\'{e}gularit\'{e} des solutions de l'\'{e}quation des milieux poreux
  dans {${\bf R}^{N}$}.
\newblock {\em C. R. Acad. Sci. Paris S\'{e}r. A-B}, 288(2):A103--A105, 1979.

\bibitem{AG}
{Aronson, D. G. and Graveleau, J.}
\newblock A selfsimilar solution to the focusing problem for the porous medium
  equation.
\newblock {\em European Journal of Applied Mathematics}, 4(1):65–81, 1993.

\bibitem{BCGR}
{Bresch, D. and Colin, T. and Grenier, E. and Ribba, B. and Saut, O.}
\newblock Computational modeling of solid tumor growth: the avascular stage.
\newblock {\em SIAM J. Sci. Comput.}, 32(4):2321--2344, 2010.

\bibitem{BPPS}
{Bubba, F. and Perthame, B. and Pouchol, C. and Schmidtchen, M.}
\newblock Hele-{S}haw limit for a system of two reaction-(cross-)diffusion
  equations for living tissues.
\newblock {\em Arch. Rational. Mech. Anal.}, 236:735--766, 2020.

\bibitem{byrneChaplain96}
{Byrne, H. M. and Chaplain, M. A. J.}
\newblock Growth of necrotic tumors in the presence and absence of inhibitors.
\newblock {\em Mathematical biosciences}, 135(2):187--216, 1996.

\bibitem{BC96}
{Byrne, H. M. and Chaplain, M. A. J.}
\newblock Modelling the role of cell-cell adhesion in the growth and
  development of carcinomas.
\newblock {\em Math. Comput. Modelling}, 24:1--17, 1996.

\bibitem{BD}
{Byrne, H. M. and Drasdo, D.}
\newblock Individual-based and continuum models of growing cell populations: a
  comparison.
\newblock {\em J. Math. Biol.}, 58(4-5):657--687, 2009.

\bibitem{BKMP2}
{Byrne, H. M. and King, J. R. and McElwain, D. L. S. and Preziosi, L.}
\newblock A two-phase model of solid tumour growth.
\newblock {\em Appl. Math. Lett.}, 16(4):567--573, 2003.

\bibitem{BP}
{Byrne, H. M. and Preziosi, L.}
\newblock Modelling solid tumour growth using the theory of mixtures.
\newblock {\em Math. Med. Biol.}, 20(4):341--366, 2004.

\bibitem{SalsaCaffa}
{Caffarelli, L. A. and Salsa, S.}
\newblock {\em A geometric approach to free boundary problems}, volume~68 of
  {\em Graduate Studies in Mathematics}.
\newblock American Mathematical Society, Providence, RI, 2005.

\bibitem{CFSS}
{Carrillo, J. A. and Fagioli, S. and Santambrogio, F. and Schmidtchen, M.}
\newblock Splitting schemes and segregation in reaction cross-diffusion
  systems.
\newblock {\em SIAM J. Math. Anal.}, 50(5):5695--5718, 2018.

\bibitem{CBCB}
{Chatelain, C. and Balois, T. and Ciarletta, P. and Ben Amar, M.}
\newblock Emergence of microstructural patterns in skin cancer: a phase
  separation analysis in a binary mixture.
\newblock {\em New Journal of Physics}, 13:115013, 2011.

\bibitem{CrPi1982}
{Crandall, M. G. and Pierre M.}
\newblock Regularizing effects for {$u_{t}=\Delta \varphi (u)$}.
\newblock {\em Trans. Amer. Math. Soc.}, 274(1):159--168, 1982.

\bibitem{DeSc}
{D{\c e}biec, T. and Schmidtchen, M.}
\newblock Incompressible limit for a two-species tumour model with coupling
  through brinkman's law in one dimension.
\newblock {\em Acta Applicandae Mathematicae}, 2020.

\bibitem{DHV}
{Degond, P. and Hecht, S. and Vauchelet, N.}
\newblock Incompressible limit of a continuum model of tissue growth for two
  cell populations.
\newblock {\em Networks \& Heterogeneous Media}, 15(1):57--85, 2020.

\bibitem{HS}
{Figalli, A. and Shahgholian, H.}
\newblock An overview of unconstrained free boundary problems.
\newblock {\em Philos. Trans. Roy. Soc. A}, 373(2050):20140281, 11, 2015.

\bibitem{frihie}
{Friedman, A.}
\newblock A hierarchy of cancer models and their mathematical challenges.
\newblock {\em Discrete Contin. Dyn. Syst. Ser. B}, 4(1):147--159, 2004.
\newblock Mathematical models in cancer (Nashville, TN, 2002).

\bibitem{Fri}
{Friedman, A.}
\newblock Mathematical analysis and challenges arising from models of tumor
  growth.
\newblock {\em Math. Models Methods Appl. Sci.}, 17(suppl.):1751--1772, 2007.

\bibitem{Green}
{Greenspan, H. P.}
\newblock On the growth and stability of cell cultures and solid tumors.
\newblock {\em J. Theoret. Biol.}, 56(1):229--242, 1976.

\bibitem{GPS}
{Gwiazda, P. and Perthame, B. and Świerczewska-Gwiazda, A.}
\newblock A two-species hyperbolic–parabolic model of tissue growth.
\newblock {\em Communications in Partial Differential Equations},
  44(12):1605--1618, 2019.

\bibitem{LTWZ}
{Liu, J.-G. and Tang, M. and Wang L. and Zhou, Z.}
\newblock An accurate front capturing scheme for tumor growth models with a
  free boundary limit.
\newblock {\em J. Comp. Phys.}, 364:73--94, 2018.

\bibitem{LLP}
{Lorenzi, T. and Lorz, A. and Perthame, B.}
\newblock On interfaces between cell populations with different mobilities.
\newblock {\em Kinet. Relat. Models}, 10(1):299--311, 2017.

\bibitem{LFJCMWC}
{Lowengrub, J. S. and Frieboes, H. B. and Jin, F. and Chuang, Y.-L. and Li, X
  and Macklin, P. and Wise, S. M. and Cristini, V.}
\newblock Nonlinear modelling of cancer: bridging the gap between cells and
  tumours.
\newblock {\em Nonlinearity}, 23(1):R1--R91, 2010.

\bibitem{VaVi}
{Lu, P. and Ni, L. and V\'{a}zquez, J. L. and Villani, C.}
\newblock An accurate front capturing scheme for tumor growth models with a
  free boundary limit.
\newblock {\em J. Comput. Phys.}, 364:73--94, 2018.

\bibitem{MMACCL}
{Macklin, P. and McDougall, S. and Anderson, A. R. and Chaplain, M. A. and
  Cristini, V. and Lowengrub, J.}
\newblock Multiscale modelling and nonlinear simulation of vascular tumour
  growth.
\newblock {\em J. Math. Biol.}, 58(4-5):765--798, 2009.

\bibitem{MePeQu}
{Mellet, A. and Perthame, B. and Quir\'{o}s, F.}
\newblock A {H}ele-{S}haw problem for tumor growth.
\newblock {\em J. Funct. Anal.}, 273(10):3061--3093, 2017.

\bibitem{PQV}
{Perthame, B. and Quir\'{o}s, F. and V\'{a}zquez, J. L.}
\newblock The {H}ele-{S}haw asymptotics for mechanical models of tumor growth.
\newblock {\em Arch. Ration. Mech. Anal.}, 212(1):93--127, 2014.

\bibitem{PTV}
{Perthame, B. and Tang, M. and Vauchelet, N.}
\newblock Traveling wave solution of the {H}ele-{S}haw model of tumor growth
  with nutrient.
\newblock {\em Math. Models Methods Appl. Sci.}, 24(13):2601--2626, 2014.

\bibitem{PV}
{Perthame, B. and Vauchelet, N.}
\newblock Incompressible limit of a mechanical model of tumour growth with
  viscosity.
\newblock {\em Philos. Trans. Roy. Soc. A}, 373(2050):20140283, 16, 2015.

\bibitem{PT}
{Preziosi, L. and Tosin, A.}
\newblock Multiphase modelling of tumour growth and extracellular matrix
  interaction: mathematical tools and applications.
\newblock {\em J. Math. Biol.}, 58(4-5):625--656, 2009.

\bibitem{RJPJ}
{Ranft, J. and Basana, M. and Elgeti, J. and Joanny, J. F. and Prost, J. and
  J\"ulicher, F.}
\newblock Fluidization of tissues by cell division and apoptosis.
\newblock {\em Natl. Acad. Sci. USA}, 49:657--687, 2010.

\bibitem{RSCB}
{Ribba, B. and Saut, O. and Colin, T. and Bresch, D. and Grenier, E. and
  Boissel, J. P.}
\newblock A multiscale mathematical model of avascular tumor growth to
  investigate the therapeutic benefit of anti-invasive agents.
\newblock {\em J. Theoret. Biol.}, 243(4):532--541, 2006.

\bibitem{RCM}
{Roose, T. and Chapman, S. J. and Maini, P. K.}
\newblock Mathematical models of avascular tumor growth.
\newblock {\em SIAM Rev.}, 49(2):179--208, 2007.

\bibitem{SCh}
{Sherratt, Jonathan A. and Chaplain, Mark A. J.}
\newblock A new mathematical model for avascular tumour growth.
\newblock {\em J. Math. Biol.}, 43(4):291--312, 2001.

\bibitem{V}
{Vazquez, J. L.}
\newblock {\em The porous medium equation: mathematical theory}.
\newblock Oxford Mathematical Monographs. The Clarendon Press, Oxford
  University Press, Oxford, 2007.
\newblock Mathematical theory.

\end{thebibliography}
\end{document}